\documentclass[preprint,12pt]{article}

\usepackage{amsmath,amsthm,amscd,amssymb,latexsym,upref,stmaryrd,epsfig,graphicx}
\usepackage[cp1251]{inputenc}
\usepackage[english]{babel}
\usepackage{mathtext}
\usepackage{amsfonts}
\usepackage {tikz}
\usepackage{hyperref}
\usepackage[numbers,sort&compress]{natbib}
\newcommand*{\mailto}[1]{\href{mailto:#1}{\nolinkurl{#1}}}
\usepackage{amsmath}
\usepackage{amsthm}

\newtheorem{theorem}{Theorem}[section]
\newtheorem{lemma}{Lemma}[section]
\newtheorem{remark}{Remark}[section]

\numberwithin{equation}{section}

\begin{document}
\thispagestyle{empty}
{\Large\bf   Several generalized Bohr-type inequalities with two }\\\\
{\centerline {\Large\bf parameters} }\\\\
\begin{center}
{\bf Wanqing Hou  \quad  Qihan Wang  \quad Boyong Long$^*$}
\let\thefootnote\relax\footnotetext{\\This work is supported by NSFC (No.12271001) and Natural Science Foundation of Anhui Province (2308085MA03), China. \\$^*$Corresponding author.\\ \quad Email: wanqinghou99@163.com;\quad qihan@ahu.edu.cn; \quad boyonglong@163.com
}
\end{center}

\centerline {\small  (School of Mathematical Sciences, Anhui University, Hefei 230601, China)}
 \hspace*{\fill} \\\\
\noindent{\bf Abstract:} {In this paper, several Bohr-type inequalities are generalized to the form with two parameters for the bounded analytic function. Most of the results are sharp.
 }



\section{Introduction}

Let $\mathbb{D}:= \{z\in\mathbb{C}: |z|<1\}$ be the open unit disk and  $\mathbb{\mathcal{A}}:=\{f(z): f(z)=\sum_{k=0}^\infty a_kz^k, z\in\mathbb{D}\}$ be the class of analytic functions on $\mathbb{D}$. If $|f(z)| < 1$ for $z\in\mathbb{D}$, then
\begin{equation}\label{1.1}
\sum_{n=0}^{\infty}|a_{n}||z|^{n}\leq1 \quad for \quad r\leq1/3,
\end{equation}
where $r=|z|$, and  the constant $1/3$ is the best possible.
The inequality(\ref{1.1})is known as the Bohr's inequality, the constant $1/3$ is called the Bohr radius. At the beginning, Bohr\cite{bohr1914theorem} proved the inequality for $r\leq1/6$ in 1914. Later, the sharp bound $r\leq1/3$ was proved by Riesz, Schur and Wiener independently(\cite{sidon1927satz,tomic1962theoreme}).

 Ali et al.\cite{ali2017note} considered Bohr's phenomenon for the classes of odd analytic functions.
 Kayumov and Ponnusamy   generalized the classical Bohr's inequality in\cite{kayumov2018improved} and explored  Bohr type inequalities for analytic functions with Lacunary series in\cite{kayumov2018bohr}.

Bohr's phenomenon and Bohr's inequality of harmonic functions were studied in \cite{evdoridis2019improved, evdoridis2021improved, allu2021bohr,liu2020bohr}. Furthermore,
 The same problem has been discussed for  several complex variables functions, see
\cite{aizenberg2000multidimensional, liu2021multidimensional, defant2006logarithmic, boas1997bohr}.
 In \cite{kumar2021bohr}, for complex integral operators of bounded analytic functions defined on unit disk $\mathbb{D}$, the  Bohr-type  phenomenon are considered.

At the beginning,  the Bohr radius were researched for mappings whose image domain bounded in  the unit disk $\mathbb{D}$.  Later,
it was discussed for classes of mappings whose image domain bounded in  other domains, for example, in convex domain (\cite{muhanna2011bohr's,muhanna2013bohr}),
wedge-domain  ( \cite{ali2017note}),  and
concave-wedge domain (\cite{muhanna2014bohr}).

In recent years, some bohr-type inequalities with one parameter for bounded analytic functions  were discussed, see  \cite{wu2022some, hu2021bohr, hu2022bohr, MR4257988}.\\

The following are some known results.
\begin{theorem}\label{theorem1.1}\rm{\cite{kayumov2017bohrrogosinski}}
Suppose that $f(z)=\sum_{k=0}^{\infty}a_{k}z^{k}$ is analytic in $\mathbb{D}$ and $|f(z)|<1$. Then
$$|f(z)|+\sum_{k=1}^{\infty}|a_{k}||z|^{k} \leq 1 \quad for \quad |z|=r\leq \sqrt{5}-2$$
and the radius $\sqrt{5}-2$ is the best possible. Moreover,
$$|f(z)|^2+\sum_{k=1}^{\infty}|a_{k}||z|^{k} \leq 1 \quad for \quad |z|=r\leq \frac{1}{3}$$
and the radius $\frac{1}{3}$ is the best possible.
\end{theorem}

\begin{theorem}\label{theorem1.2}\rm{\cite{liu2018bohr}}
Suppose that $N(\geq2)$ is an integer, $f(z)=\sum_{k=N}^{\infty}a_{k}z^{k}$ is analytic in $\mathbb{D}$ and $|f(z)|<1$. Then
$$|f(z)|+\sum_{k=N}^{\infty}|\frac{{f^{(k)}(z)}}{k!}|r^{k}\leq 1$$
for $|z|=r\leq R_{N}$,
where $R_{N}$ is the minimum positive root of the equation $(1+r)(1-2r)(1-r)^{N-1}-2r^{N}=0$. The radius $R_{N}$ is the best possible.
\end{theorem}

\begin{theorem}\label{theorem1.3}\rm{\cite{wu2022some}}
Suppose that $f(z)$ is analytic in $\mathbb{D}$ and $|f(z)| < 1$. Then for arbitrary $\lambda\in (0, +\infty)$, it holds that
\begin{equation*}
|f(z)|+\lambda \sum_{k = 1}^\infty  |a_{k}|r^{k}  \leq 1
\quad
for \quad r \leq R,
\end{equation*}
where $R$ is the unique  root in $(0,1)$ of the equation
$$2r\lambda(1+r)-(1-r)^2=0$$
and the radius $R$ is the best possible.
\end{theorem}

In this paper,  the above Theorem \ref{theorem1.1}, \ref{theorem1.2} and \ref{theorem1.3}  are generalized to the form with two parameters.
The introduction of two parameters increase the difficulty of the corresponding problems,  especially to find the sharp bound of the Bohr-type radii. Sometimes, we conjecture  a radius is sharp. But we can not to verify the corresponding extreme functions.  Even so,  some of our results in this paper are still sharp.

\section{ Some Lemmas}
In order to prove our main results, we need the following lemmas at first.
\vskip 0.15in
\begin{lemma}\rm(Schwarz-Pick lemma)  Let $\phi(z) $ be analytic in the unit disk $\mathbb{D}$ and $|\phi(z)|<1$. Then
$$\frac{|\phi(z_{1})-\phi(z_{2})|}{|1-\overline{\phi(z_{1})}\phi(z_{2})|}\leq \frac{|z_{1}-z_{2}|}{|1-\overline{z_{1}}z_{2}|}\quad for\quad z_{1},z_{2}\in\mathbb{D},$$
and equality holds for distinct $z_{1},z_{2}\in \mathbb{D}$ if and only if $\phi$ is a M\"{o}bius transformation.
In particularly,
$$|\phi'(z)|\leq\frac{1-|\phi(z)|^{2}}{1-|z|^{2}}\quad for \quad z\in\mathbb{D},$$
and equality holds for some $z\in \mathbb{D}$ if and only if $\phi$ is a M\"{o}bius transformation.
\end{lemma}
\begin{lemma}\rm\label{lemma2.2} {(\cite{MR2017933})} { Suppose $f(z)=\sum_{n=0}^{\infty}a_{n}z^{n}$ is analytic in the unit disk $\mathbb{D}$ and $|f(z)|\leq1$. Then $|a_{n}|\leq1-|a_{0}|^2$ for all $n\in\mathbb{N}$.}
\end{lemma}
\begin{lemma}\rm\label{lemma2.3} Let $p\in (0,2]$, $A\in(0, p/2]$ and $x\in [0,1)$. Then
\begin{equation*}\label{2.1}
\Phi'(x)>0 \quad where \quad \Phi(x)=x^{p}+A(1-x^2).
\end{equation*}

The proof is simple. We omit it.
\end{lemma}

\begin{lemma}\rm\label{lemma2.4}{\rm(\cite{dai2008note})} { Suppose $f(z)$ is analytic in the unit disk $\mathbb{D}$ and $|f(z)|\leq1$. Then for all $k\in\mathbb{N}$, it holds that $|f^{(k)}(z)|\leq\frac{(1-|f(z)|^2)k!}{(1-|z|^2)^k}(1+|z|)^{k-1}$ for $|z|<1$.}
\end{lemma}

\begin{lemma}\rm\label{lemma2.5} Let $x \geq y>0$. Then  it holds that
\begin{equation}\label{2.1}
 \alpha x^{\alpha-1}(x-y)\leq x^{\alpha}-y^{\alpha}
\end{equation}   for $0<\alpha \leq 1$  and
\begin{equation}\label{2.2}
 \alpha y^{\alpha-1}(x-y)\leq x^{\alpha}-y^{\alpha}
\end{equation} for $1<\alpha$.

\begin{proof} We divide it into several cases to discuss.

Case 1: If $x=y$ or $\alpha=1$.  It is obvious.

Case 2: If $y<x$.   Let $g(t)=t^{\alpha}, \quad t\in (0,\infty)$.
Then Mean-Value Theorem  implies that there exists $\xi\in (y, x)$, such that
\begin{align}\label{2.4}
g(x)-g(y)=x^{\alpha}-y^{\alpha}=\alpha \xi^{\alpha-1}(x-y).
\end{align}

If $\alpha\in(0,1)$, then $x^{\alpha-1}<\xi^{\alpha-1}$.
It follows that  \begin{align*}
\alpha x^{\alpha-1}(x-y)< x^{\alpha}-y^{\alpha}.
\end{align*}

If $\alpha\in(1,\infty)$, then  $y^{\alpha-1}<\xi^{\alpha-1}$. Thus equality (\ref{2.4}) implies that
\begin{align*}
\alpha y^{\alpha-1}(x-y)< x^{\alpha}-y^{\alpha}.
\end{align*}
\end{proof}
\end{lemma}

\begin{lemma}\rm\label{lemma2.6} It holds that
$$\left\{ \begin{array}{*{20}{c}}
{r_{1,t}<r_{2,t} }, \quad for \quad t \in ( 0, t_{*})\\
{r_{1,t}>r_{2,t} },  \quad for \quad t \in  (t_{*}, 1)
\end{array}
\right.$$
where $r_{1,t}=\frac{t-2+\sqrt{9t^2-4t+4}}{4t}$, $r_{2,t}=\frac{4-t-\sqrt{t^2-8t+8}}{4}$ and $t_{*}$ is the unique positive real root of the equation $$t^4-8t^3+7t^2+4t-4=0$$ for $t\in (0,1)$.
\begin{proof}
If $r_{1,t}-r_{2,t}<0$, then we have \begin{align*}\sqrt{9t^2-4t+4}+t\sqrt{t^2-8t+8}<-t^2+3t+2,\end{align*} that is \begin{align*}\sqrt{9t^4-76t^3+108t^2-64t+32}<t^2-6t+8.\end{align*}
It follows that $$t^4-8t^3+7t^2+4t-4<0.$$
Let $$h(t)=t^4-8t^3+7t^2+4t-4.$$
It is sufficient to show that $h(r)<0$ holds for $t\in(0,t_{*})$.
It follows that \begin{align*}
h'(t)=&4t^3-24t^2+14t+4,\\
h'''(t)=&24t-48<0.
\end{align*}
Observe that  $h'(0)>0$, $h'(1)<0$  and $h'(t)$ is strictly convex for $t\in(0,1)$. It implies that there exists a unique $t_0\in(0,1)$, such that $h'(t)>0$   for $t\in(0,t_0)$ and $h'(t)<0$   for $t\in(t_0, 1)$.  Because of $h(0)<0$ and $h(1)=0$,  there is a unique $t_*\in(0, t_0)$ such that $h(t_*)=0$. Hence, $h(t)<0$ holds for $t\in(0,t_{*})$.

The rest of the proof is easy. We omit it.
\end{proof}
\end{lemma}

\section{ Main results }
\begin{theorem}\rm
\label{theorem3.1} Suppose that $f \in \mathcal{A}$ with $|f(z)| < 1$. Let $a:=|a_{0}|$. Then for $(p,\lambda)\in\uppercase\expandafter{\romannumeral1}\cup\uppercase\expandafter{\romannumeral2}$, it holds that
\begin{equation}\label{3.1}
|f(z)|^{p}+\lambda\sum_{k=1}^{\infty}|a_{k}||z|^{k}\leq 1
\end{equation} for$|z|=r\leq R_{\lambda,p}$, where $$R_{\lambda,p}=\left\{ \begin{array}{*{20}{c}}
{\frac{-p-\lambda+\sqrt{\lambda^2+4p\lambda}}{2\lambda-p},} \quad for \quad p\neq2\lambda\\
{1/3,}  \quad for \quad p=2\lambda,
\end{array}
\right.$$
$I=\left\{(p,\lambda)|0<p\leq1, 0<\lambda\right\}$ and $II=\left\{(p,\lambda)|1<p\leq\frac{3}{2}, 0<\lambda\leq\frac{9p}{10}\right\}$.
Furthermore, the radius $R_{\lambda,p}$ is sharp.
\end{theorem}

\begin{proof}
According to the assumption and the Schwarz-Pick lemma, we have
\begin{equation}\label{3.2}
|f(z)|\leq\frac{|z|+a}{1+a|z|} \quad for \quad z\in\mathbb{D}.
\end{equation}
Using inequality (\ref{3.2}) and Lemma \ref{lemma2.2}, we have \begin{align*}
|f(z)|^{p}+\lambda\sum_{k=1}^{\infty}|a_{k}||z|^{k}\leq&
(\frac{r+a}{1+ar})^{p}+\lambda(1-a^2)\frac{r}{1-r}:=A(a,r).
\end{align*}
Now, we need to show that $A(a,r)\leq1$ holds for $r\leq R_{\lambda,p}$. It is equivalent to show $B(a,r)\leq0$, where
\begin{align}
B(a,r)=&(a+r)^{p}(1-r)+\lambda(1-a^2)r(1+ar)^{p}-(1+ar)^{p}(1-r)\nonumber\\
\label{3.3}=&(1-r)[(a+r)^{p}-(1+ar)^{p}]+\lambda(1-a^2)r(1+ar)^{p}.
\end{align}
Next, we divide it into two cases to discuss.

{\bf Case 1} $(p,\lambda)\in\uppercase\expandafter{\romannumeral1}$, where $I=\left\{(p,\lambda)|0<p\leq1, 0<\lambda\right\}$.
According to Lemma \ref{lemma2.5}, we have
\begin{align*}
B(a,r)
\leq&(-p)(1+ar)^{p-1}(1-a)(1-r)^2+\lambda(1-a^2)r(1+ar)^{p}\\
:=&(1-a)(1+ar)^{p-1}g(a),
\end{align*}
where $g(a)=\lambda(1+a)r(1+ar)-p(1-r)^2$.
Then we have $g(a)\leq g(1)$ holds for $a<1$.
Let $$h(r)=2r\lambda(1+r)-p(1-r)^2.$$
It follows that $g(1)=h(r)$.
It is enough to show that $h(r)\leq 0$ holds for $r\leq R_{\lambda,p}$.
Direct computations lead to\begin{align*}
h'(r)=&2\lambda(1+r)+2r\lambda+2p(1-r)>0,
\end{align*}
$h(0)=-p<0$ and $h(1)=4\lambda>0$. Then we have $h(\frac{-p-\lambda+\sqrt{\lambda^2+4p\lambda}}{2\lambda-p})=0$ for $p\neq2\lambda$ and $h(1/3)=0$ for $p=2\lambda$. Thus there is a unique number $R_{\lambda,p}$ such that $h(R_{\lambda,p})=0$. Therefore, $h(r)\leq0$ holds for $r\leq R_{\lambda,p}$.

{\bf Case 2}   $(p,\lambda)\in\uppercase\expandafter{\romannumeral2}$, where $II=\left\{(p,\lambda)|1<p\leq\frac{3}{2}, 0<\lambda\leq\frac{9p}{10}\right\}$.
According to Lemma \ref{lemma2.5}, equation (\ref{3.3}) leads to
\begin{align*}
B(a,r)
\leq&(-p)(a+r)^{p-1}(1-a)(1-r)^2+\lambda(1-a^2)r(1+ar)^{p}\\
:=&(1-a)\varphi(a,r),
\end{align*}
where $$\varphi(a,r)=\lambda(1+a)r(1+ar)^{p}-p(a+r)^{p-1}(1-r)^2.$$
It is enough to show that $\varphi(a,r)\leq 0$ holds for $r\leq R_{\lambda,p}$.
Easy computations yield that \begin{align*}
\frac{\partial^{2} \varphi(a,r)}{\partial a^{2}}&=2\lambda(1+ar)^{p-1}pr^2+\lambda(1+ar)^{p-2}(1+a)p(p-1)r^3\\
&+p(p-1)(2-p)(1-r)^2(a+r)^{p-3}>0.
\end{align*}
It means that for fixed $r$,  the function $\varphi(a,r)$ is strictly convex for $a\in[0,1)$. It follows that
\begin{align*}
\varphi(1,r)&=(1+r)^{p-1}[2\lambda r(1+r)-p(1-r)^2]=(1+r)^{p-1}h(r),\\
\varphi(0,r)&=r[\lambda-pr^{p-2}(1-r)^2]:=r\psi(r).
\end{align*}
So, we just  need to prove $h(r)\leq0$ and $\psi(r)\leq0$ for $r\leq R_{\lambda,p}$.

It is
obvious that $h(r)\leq0$ holds for $r\leq R_{\lambda,p}$. Next, we focus on proving the inequality $\psi(r)\leq0$ for $r\leq R_{\lambda,p}$.

Observe that  $\psi(r)$ is an increasing function for $r\in (0,1)$.  Thus, we just need to show that $\psi(R_{\lambda,p})\leq0$.
Since  $h(R_{\lambda,p})=0$, we have $p(1-R_{\lambda,p})^2=2\lambda R_{\lambda,p}(1+R_{\lambda,p})$.  Replacing $p(1-R_{\lambda,p})^2$ with $2\lambda R_{\lambda,p}(1+R_{\lambda,p})$ in $\psi(R_{\lambda,p})$,
We have $$\psi(R_{\lambda,p})=\lambda-pR_{\lambda,p}^{p-2}(1-R_{\lambda,p})^2=\lambda[1-2R_{\lambda,p}^{p-1}(1+R_{\lambda,p})].$$
Simple computation leads to that $R_{\lambda,p}^{p-1}(1+R_{\lambda,p})\geq2R_{\lambda,p}^{p-\frac{1}{2}}\geq2R_{\lambda,p}$ for $1<p\leq\frac{3}{2}$.
Thus, $$\psi(R_{\lambda,p})\leq\lambda(1-4R_{\lambda,p}).$$
Now, it is sufficient for us to show that $1-4R_{\lambda,p}\leq0$, that is $R_{\lambda,p}\geq\frac{1}{4}$.
We divide it into three cases to discuss.

(i) If $0<\lambda<\frac{p}{2}$, then we have $20\lambda^2-28\lambda p+9p^{2}>0$. Easy computations yield that $16(\lambda^{2}+4\lambda p)<(6\lambda+3p)^{2}$ and $4\sqrt{\lambda^{2}+4\lambda p}<6\lambda+3p$, it follows that
$\frac{-p-\lambda+\sqrt{\lambda^2+4p\lambda}}{2\lambda-p}>\frac{1}{4}$. Thus, $R_{\lambda,p}>\frac{1}{4}$.

(ii) If $\lambda=\frac{p}{2}$, then we have $R_{\lambda,p}=\frac{1}{3}>\frac{1}{4}$.

(iii) If $\frac{p}{2}<\lambda\leq\frac{9p}{10}$, it is similar to that of (i), we have $R_{\lambda,p}\geq\frac{1}{4}$.

Next, we show the radius $R_{\lambda,p}$ is sharp.\\
For $a\in[0,1)$, let \begin{equation}\label{3.4}
\quad f(z)=\frac{a-z}{1-az}=a-(1-a^2)\sum_{k=1}^{\infty}a^{k-1}z^k, \quad z\in \mathbb{D}.
\end{equation}
Taking $z=-r$, then the left side of inequality (\ref{3.1}) reduces to\begin{equation}\label{3.5}
|f(-r)|^{p}+\lambda\sum_{k=1}^{\infty}|a_{k}|r^{k}
=(\frac{r+a}{1+ar})^{p}+\lambda(1-a^2)\frac{r}{1-ar}.
\end{equation}
Now we just need to show that if $r>R_{\lambda,p}$, then there exists an $a\in[0,1)$, such that the right side of (\ref{3.5}) is greater than $1$.
That is $C(a,r)>0$, where \begin{align}
C(a,r)&=(r+a)^{p}(1-ar)+\lambda(1-a^2)r(1+ar)^{p}-(1+ar)^{p}(1-ar)\nonumber\\
\label{3.6}&=(1-ar)[(a+r)^{p}-(1+ar)^{p}]+\lambda(1-a^2)r(1+ar)^{p}.
\end{align}
Now we divide it into two subcases to discuss.

{\bf Subcase 1} $(p,\lambda)\in\uppercase\expandafter{\romannumeral1}$, where $I=\left\{(p,\lambda)|0<p\leq1, 0<\lambda\right\}$.
According to Lemma \ref{lemma2.5}, we have\begin{align*}
C(a,r)
\leq&(-p)(1-ar)(1+ar)^{p-1}(1-a)(1-r)+\lambda(1-a^2)r(1+ar)^{p}\\
=&(1-a)(1+ar)^{p-1}[\lambda(1+a)r(1+ar)-p(1-ar)(1-r)]\\
\leq&(1-a)(1+ar)^{p-1}[2\lambda r(1+r)-p(1-r)^2]\\
=&(1-a)(1+ar)^{p-1}h(r).
\end{align*}
We already know that $h(r)$ is a continuous and increasing function of $r\in[0,1)$, the monotonicity of $h(r)$ leads  to that if $r>R_{\lambda,p}$, then $h(r)>0$.
By the continuity of $C(a,r)$, we claim that $$\lim_{a\rightarrow1^{-}}[C(a,r)-(1-a)(1+ar)^{p-1}h(r)]=0.$$
That is, if $r>R_{\lambda,p}$, for $\forall \, \varepsilon>0$, there exists a $\delta>0$,  when $a\in(1-\delta,1)$, we have $|C(a,r)-(1-a)(1+ar)^{p-1}h(r)|<\varepsilon$.
It follows that $C(a,r)>(1-a)(1+ar)^{p-1}h(r)-\varepsilon$. Now, taking $\varepsilon=(1-a)(1+a)^{p-1}h(r)/2$, then we have $C(a,r)>(1-a)(1+a)^{p-1}h(r)/2>0$ for $r>R_{\lambda,p}$.

{\bf Subcase 2} $(p,\lambda)\in\uppercase\expandafter{\romannumeral2}$, where $II=\left\{(p,\lambda)|1<p\leq\frac{3}{2}, 0<\lambda\leq\frac{9p}{10}\right\}$.
According to  Lemma \ref{lemma2.5} and equation(\ref{3.6}), we have
\begin{align*}
C(a,r)\leq&(-p)(1-ar)(a+r)^{p-1}(1-a)(1-r)+\lambda(1-a^2)r(1+ar)^{p}\\
=&(1-a)\mu(a,r),
\end{align*}
where $$\mu(a,r)=(1+ar)^{p}\lambda(1+a)r-p(1-ar)(1-r)(a+r)^{p-1}.$$
Observe that \begin{align*}
\mu(1,r)=(1+r)^{p-1}[2\lambda r(1+r)-p(1-r)^2]=(1+r)^{p-1}h(r).
\end{align*}
The monotonicity of $h(r)$ implies that if $r>R_{\lambda,p}$, we have $h(r)>0$.
Then, by the continuity of $C(a,r)$ and $\mu(a,r)$, we claim that $$\lim_{a\rightarrow1^{-}}[C(a,r)-(1-a)\mu(a,r)]=0.$$
That is, if $r>R_{\lambda,p}$, for $\forall \, \varepsilon>0$, there exists a $\delta>0$,  when $a\in(1-\delta,1)$, we have $|C(a,r)-(1-a)\mu(a,r)|<\varepsilon$.
It follows that $C(a,r)>(1-a)\mu(a,r)-\varepsilon$. Now, taking $\varepsilon=(1-a)\mu(a,r)/2$, then we have $C(a,r)>(1-a)\mu(a,r)/2>0$ for $r>R_{\lambda,p}$.
\end{proof}

\begin{remark}\rm  (1)
If $p=1$, then Theorem \ref{theorem3.1} reduces to Theorem \ref{theorem1.3}.\\
(2) If $p=1$ and $\lambda=1$, then Theorem \ref{theorem3.1} reduces to Theorem \ref{theorem1.1}.
\end{remark}

\begin{theorem}\rm
\label{theorem3.2} Suppose that $f \in \mathcal{A}$ and $|f(z)| < 1$. Let $a:=|a_{0}|$. Then for $\lambda\in (0,\infty)$ and $p\in (0,1]$, we have
\begin{equation}\label{3.7}
|f(z)|^{p}+\lambda\sum_{k=1}^{\infty}|\frac{{f^{(k)}(z)}}{k!}||z|^{k}\leq 1
\end{equation}
for $|z|=r\leq r_{\lambda, p}<\frac{1}{2}$,
where $r_{\lambda, p}=\frac{-2\lambda-p+\sqrt{9p^2+4\lambda p+4\lambda^2}}{4p}$.
Furthermore, the radius $r_{\lambda,p}$ is the best possible.
\end{theorem}

\begin{proof}
According to the assumption and Lemma \ref{lemma2.4}, we have
\begin{align*}
|f(z)|^{p}+\lambda\sum_{k=1}^{\infty}|\frac{{f^{(k)}(z)}}{k!}||z|^{k}
\leq&|f(z)|^{p}+\lambda(1-|f(z)|^2)\sum_{k=1}^{\infty}\frac{r^{k}}{(1-r)^{k}(1+r)}\\
=&|f(z)|^{p}+\frac{\lambda r}{(1+r)(1-2r)}(1-|f(z)|^{2}).
\end{align*}

Next, in order to use  Lemma \ref{lemma2.3}, we need to verify the condition  $\frac{\lambda r}{(1+r)(1-2r)}\leq\frac{p}{2}$. It is equivalent to show that $\beta(r)\leq0$, where $$\beta(r)=2pr^2+(2\lambda+p)r-p.$$
Observe that\begin{align*}
\beta'(r)=&4pr+2\lambda+p>0,
\end{align*}
$\beta(0)=-p<0$ and $\beta(\frac{1}{2})=\lambda>0$. There is a unique number $r_{\lambda, p}=\frac{-2\lambda-p+\sqrt{9p^2+4\lambda p+4\lambda^2}}{4p}\in (0,\frac{1}{2})$ such that $\beta(r_{\lambda, p})=0$. Hence, $\beta(r)\leq0$ holds for $r\leq r_{\lambda, p}$.

Now Lemma \ref{lemma2.3} implies that $|f(z)|^{p}+\frac{r\lambda}{(1+r)(1-2r)}(1-|f(z)|^{2})$ is a monotonically increasing function of $|f(z)|$.  Considering inequality (\ref{3.2}), we have
\begin{align*}
|f(z)|^{p}+\frac{\lambda r}{(1+r)(1-2r)}(1-|f(z)|^{2})
\leq&(\frac{r+a}{1+ar})^{p}+\frac{\lambda r}{(1+r)(1-2r)}[1-(\frac{r+a}{1+ar})^{2}]:=F(a,r).
\end{align*}
It is sufficient for us to show that $F(a,r)\leq1$ holds for $r\leq r_{\lambda, p}$.  That is $G(a,r)\leq0$ for $r\leq r_{\lambda, p}$, where
\begin{align*}
G(a,r)
=&(a+r)^{p}(1-2r)+\lambda(1-a^2)r(1-r)(1+ar)^{p-2}-(1+ar)^{p}(1-2r)\nonumber\\
=&(1-2r)[(a+r)^{p}-(1+ar)^{p}]+\lambda(1-a^2)r(1-r)(1+ar)^{p-2}.
\end{align*}
Using Lemma \ref{lemma2.5}, we have \begin{align*}
G(a,r)
\leq&(-p)(1-2r)(1+ar)^{p-1}(1-a)(1-r)+\lambda(1-a^2)r(1-r)(1+ar)^{p-2}\\
=&(1-a)(1+ar)^{p-2}(1-r)\alpha(a),
\end{align*}
where $$\alpha(a)=\lambda(1+a)r-p(1-2r)(1+ar).$$
It is enough to show that $\alpha(a)\leq 0$ holds for $r\leq r_{\lambda, p}$.
Observe that\begin{align}\label{3.8}
\alpha'(a)=&2pr^2+r(\lambda-p).
\end{align}

If $\lambda \in (0,p)$, equation (\ref{3.8}) implies that $\alpha(a)$ is  decreasing for $r\in(0,\frac{p-\lambda}{2p})$ and  increasing  for $r\in[\frac{p-\lambda}{2p},\frac{1}{2})$. It follows that $\alpha(a)\leq \alpha(0)=(\lambda+2p)r-p<(\lambda+2p)\frac{p-\lambda}{2p}-p<0$ for $r\in(0,\frac{p-\lambda}{2p})$ and $\alpha(a)\leq \alpha(1)=\beta(r)$ for $r\in[\frac{p-\lambda}{2p},r_{\lambda, p})$. Because   $\beta(r)\leq0$ holds for $r\leq r_{\lambda, p}$, we have $\alpha(a)\leq 0$ holds for $r\leq r_{\lambda, p}$.

If $\lambda\in[p,\infty)$, by equation (\ref{3.8}), we have $\alpha'(a)>0$. It follows that $\alpha(a)\leq \alpha(1)=2pr^2+(2\lambda+p)r-p=\beta(r)$. Since $\beta(r)\leq0$ holds for $r\leq r_{\lambda, p}$, we have $\alpha(a)\leq 0$ holds for $r\leq r_{\lambda, p}$.

Next, we show the radius $r_{\lambda, p}$ is the best possible.  Considering the function (\ref{3.4}) and with $z=r$, then the left side of inequality (\ref{3.7}) reduces to
\begin{align}
|f(r)|^{p}+\lambda\sum_{k=1}^{\infty}|\frac{{f^{(k)}(r)}}{k!}|r^{k}\nonumber
&=(\frac{a-r}{1-ar})^{p}+\lambda\sum_{k=1}^{\infty}\frac{a^{k-1}(1-a^2)r^k}{(1-ar)^{k+1}}\nonumber\\
&=(\frac{a-r}{1-ar})^{p}+\lambda(1-a^2)\frac{r}{(1-ar)(1-2ar)}.\label{3.9}
\end{align}
Now we just need to show that if $r >r_{\lambda, p}$, then there exists an $a \in [0,1)$, such that (\ref{3.9}) is greater than 1. That is $Q(a,r)>0$ for $r > r_{\lambda, p}$, where \begin{align*}
Q(a,r)=&(a-r)^{p}(1-2ar)+\lambda(1-a^2)r(1-ar)^{p-1}-(1-ar)^{p}(1-2ar)\\
=&(1-2ar)[(a-r)^{p}-(1-ar)^{p}]+\lambda(1-a^2)r(1-ar)^{p-1}.\end{align*}
By Lemma \ref{lemma2.5}, we have\begin{align*}
Q(a,r)
\leq&(-p)(1-2ar)(1-ar)^{p-1}(1-a)(1+r)+\lambda(1-a^2)r(1-ar)^{p-1}\\
=&(1-a)(1-ar)^{p-1}[\lambda(1+a)r-p(1-2ar)(1+r)]\\
\leq&(1-a)(1-ar)^{p-1}[2\lambda r-p(1-2r)(1+r)]\\
=&(1-a)(1-ar)^{p-1}\beta(r).
\end{align*}

We already know that $\beta(r)$ is a continuous and increasing function of $r\in[0,\frac{1}{2})$. The monotonicity of $\beta(r)$ leads to that if $r>r_{\lambda,p}$, then $\beta(r)>0$.
By the continuity of $Q(a,r)$, we have that $$\lim_{a\rightarrow1^{-}}[Q(a,r)-(1-a)(1-ar)^{p-1}\beta(r)]=0.$$
That is, if $r>r_{\lambda,p}$, for $\forall \,\varepsilon >0$, there exists a $\delta>0$, when $a\in(1-\delta,1)$, we have $|Q(a,r)-(1-a)(1-ar)^{p-1}\beta(r)|<\varepsilon$.
It follows that $Q(a,r)>(1-a)(1-ar)^{p-1}\beta(r)-\varepsilon$. Now, taking $\varepsilon=\frac{1-a}{2}(1-ar_{\lambda,p})^{p-1}\beta(r)$,  we have $Q(a,r)>\frac{1-a}{2}(1-ar_{\lambda,p})^{p-1}\beta(r)>0$ for $r>r_{\lambda,p}$.
\end{proof}
\begin{remark}\rm
If $p=1$ and $\lambda =1$, then Theorem \ref{theorem3.2} reduces to Theorem \ref{theorem1.2}.
\end{remark}

\begin{theorem}\rm
\label{theorem3.3} Suppose that $f \in \mathcal{A}$ and $|f(z)| < 1$. Let $a:=|a_{0}|$. Then for $t \in (0,1)$ and $p\in (0,1]$, we have
\begin{equation}\label{3.10}
t|f(z)|^{p}+(1-t)\sum_{k=0}^{\infty}|a_{k}||z|^{k}\leq 1
\end{equation}
for $|z|=r \leq R_{t,p}$, where $$R_{t,p}=\left\{ \begin{array}{*{20}{c}}
{\frac{1-t+tp-2\sqrt{(1-t)(tp+1-t)}}{tp-3+3t},} \quad for \quad t\neq\frac{3}{p+3}\\
{1/2,}  \quad for \quad t=\frac{3}{p+3}.
\end{array}
\right.$$
Furthermore, the radius $R_{t,p}$ is the best possible.
\end{theorem}

\begin{proof}
Inequality (\ref{3.2}) and Lemma \ref{lemma2.2} lead to that
\begin{align*}
t|f(z)|^{p}+(1-t)\sum_{k=0}^{\infty}|a_{k}||z|^{k}
\leq&t(\frac{r+a}{1+ar})^{p}+(1-t)[a+(1-a^2)\frac{r}{1-r}]\\
:=&N(a).
\end{align*}
Now, it is enough to show that $N(a)\leq1$ holds for $r\leq R_{t,p}$.
It follows that
\begin{align*}
N'(a)=&tp\frac{(1-r^2)(r+a)^{p-1}}{(1+ar)^{p+1}}+(1-t)(1-\frac{2ar}{1-r}),\\
N''(a)=&tp\frac{(1-r^2)(r+a)^{p-2}}{(1+ar)^{p+2}}[p-1-2ar-(p+1)r^2]-\frac{2(1-t)r}{1-r}<0.
\end{align*}
It means that
\begin{align*}
N'(a)\geq N'(1)
=&\frac{tp(1-r)}{1+r}+\frac{(1-t)(1-3r)}{1-r}\\
:=&\frac{Q(r)}{(1+r)(1-r)},
\end{align*}
where $$Q(r)=tp(1-r)^2+(1-t)(1-3r)(1+r).$$Observe that $Q(0)=tp+1-t>0$, $Q(1)=-4(1-t)<0$, and
\begin{align*}
Q'(r)=-2tp(1-r)-2(1-t)(1+3r)\leq0.
\end{align*}
Then we have $Q(\frac{1-t+tp-2\sqrt{(1-t)(tp+1-t)}}{tp-3+3t})=0$ for $t\neq\frac{3}{p+3}$ and $h(1/2)=0$ for $t=\frac{3}{p+3}$. Thus there is a unique number $R_{t, p}$ such that $Q(r)\geq0$ for $r\leq R_{t, p}$. So, $N'(a)\geq N'(1)\geq0$ holds for $r\leq R_{t, p}$. Therefore, $N(a)\leq N(1)=1$ holds for $r\leq R_{t, p}$.

Next, we show the radius $R_{t, p}$ is the best possible.  Considering the function (\ref{3.4}) and with $z=-r$, then the left side of inequality (\ref{3.10}) reduces to
\begin{align*}
t|f(-r)|^{p}+(1-t)\sum_{k=0}^{\infty}|a_{k}|r^{k}&=t(\frac{r+a}{1+ar})^{p}+(1-t)[a+(1-a^2)\frac{r}{1-ar}]\\
&:=F(a).
\end{align*}
It is enough to show that $F(a)>1$ holds for $r> R_{t,p}$.
It follows that
\begin{align*}
F'(a)=&tp\frac{(1-r^2)(r+a)^{p-1}}{(1+ar)^{p+1}}+(1-t)[1+r\frac{a^{2}r-2a+r}{(1-ar)^2}],\\
F''(a)=&tp\frac{(1-r^2)(r+a)^{p-2}}{(1+ar)^{p+2}}[p-1-2ar-(p+1)r^2]-\frac{2(1-t)r(1-r^2)}{(1-ar)^3}<0.
\end{align*}
It means that
\begin{align*}
F'(a)\geq F'(1)
=&\frac{tp(1-r)}{1+r}+\frac{(1-t)(1-3r)}{1-r}\\
=&\frac{Q(r)}{(1+r)(1-r)}.
\end{align*}
We already know that $Q(r)$ is a continuous and decreasing function of $r\in[0,1)$, the monotonicity of $Q(r)$ leads  to that if $r>R_{t,p}$, then $Q(r)<0$.
By the continuity of $F'(a)$, we claim that $$\lim_{a\rightarrow1^{-}}[F'(a)-\frac{Q(r)}{(1+r)(1-r)}]=0.$$
That is, if $r>R_{t,p}$, for $\forall \, \varepsilon>0$, there exists a $\delta>0$,  when $a\in(1-\delta,1)$, we have $|F'(a)-\frac{Q(r)}{(1+r)(1-r)}|<\varepsilon$.
It follows that $F'(a)<\frac{Q(r)}{(1+r)(1-r)}+\varepsilon$. Now, taking $\varepsilon=\frac{-Q(r)}{2(1+r)(1-r)}$, then we have $F'(a)<\frac{Q(r)}{2(1+r)(1-r)}<0$ and $F(a)\geq F(1)=1$ for $r>R_{t,p}$.

\end{proof}
\begin{remark}\rm
If $t\rightarrow0^{+}$, then Theorem \ref{theorem3.3} reduces to the case of the classical Bohr-radius problem.
\end{remark}

\begin{theorem}\rm
\label{theorem3.4} {Suppose that $f \in \mathcal{A}$ and $|f(z)| < 1$. Let $a:=|a_{0}|$. Then for $t\in (0,1)$, we have
\begin{equation*}
t|f(z)|+(1-t)\sum_{k=1}^{\infty}|\frac{{f^{(k)}(z)}}{k!}|z|^{k}+(1-t)a\leq 1
\quad
for \quad r \leq {r_t}<\frac{1}{2},
\end{equation*}
where $$r_t=\left\{ \begin{array}{*{20}{c}}
{\frac{t-2+\sqrt{9t^2-4t+4}}{4t} },  \quad &t \in ( 0, t_{*})\\
{\frac{4-t-\sqrt{t^2-8t+8}}{4}},  \quad &t \in  (t_{*}, 1)
\end{array}
 \right.$$ and $t_{*}$ is the unique positive real root of the equation $$t^4-8t^3+7t^2+4t-4=0$$ for $t\in (0,1)$.
}
\end{theorem}

\begin{proof}
According to the assumption and Lemma\ref{lemma2.4}, we have
\begin{align*}
&t|f(z)|+(1-t)\sum_{k=1}^{\infty}|\frac{{f^{(k)}(z)}}{k!}||z|^{k}+(1-t)a\\
\leq&t|f(z)|+(1-t)(1-|f(z)|^2)\sum_{k=1}^{\infty}\frac{r^{k}}{(1-r)^k(1+r)}+(1-t)a\\
=&t|f(z)|+\frac{(1-t)r}{(1+r)(1-2r)}(1-|f(z)|^{2})+(1-t)a.
\end{align*}
Next, in order to use  Lemma \ref{lemma2.3}, we need to verify the condition  $\frac{(1-t)r}{t(1+r)(1-2r)}\leq\frac{1}{2}$. It is equivalent to show that $l(r)\leq0$, where $$l(r)=2tr^2+(2-t)r-t.$$
That is $r\leq r_{1,t}$, where $r_{1,t}=\frac{t-2+\sqrt{9t^2-4t+4}}{4t}$.
Now Lemma \ref{lemma2.3} implies that $t|f(z)|+\frac{(1-t)r}{(1+r)(1-2r)}(1-|f(z)|^{2}$ is a monotonically increasing function of $|f(z)|$.  Considering inequality (\ref{3.2}), we have
\begin{align*}
&t|f(z)|+\frac{(1-t)r}{(1+r)(1-2r)}(1-|f(z)|^{2})+(1-t)a\\
\leq& t\frac{r+a}{1+ar}+\frac{(1-t)r}{(1+r)(1-2r)}[1-(\frac{r+a}{1+ar})^{2}]+(1-t)a\\
=&\frac{t(a+r)(1+ar)(1-2r)+(1-t)(1-a^2)r(1-r)+(1-t)a(1+ar)^2(1-2r)}{(1+ar)^2(1-2r)}\\
:&=K(a,r,t).
\end{align*}
It is enough to show that $K(a,r,t)\leq1$ holds for $r\leq r_t$, that is $P(a)\leq0$ holds for $r\leq r_t$, where\begin{align*}
P(a)&=(1-t)(1-2r)r^2a^3+(2r^2+tr-4r+1)ra^2+(-2tr^3+tr^2+4r^2-4r+1)a-tr^2-r^2\\
&+3r-1.
\end{align*}
Next, we divide it into two cases to discuss.

{\bf Case 1} If $t\in(0,t_{*})$, where $t_{*}$ is the unique positive real root of the equation $$t^4-8t^3+7t^2+4t-4=0$$ for $t\in (0,1)$.
Simple computations lead to that
\begin{align*}
P'(a)=&3(1-t)(1-2r)r^2a^2+2(2r^2+tr-4r+1)ra-2tr^3+tr^2+4r^2-4r+1,\\
P''(a)=&6(1-t)(1-2r)r^2a+2(2r^2+tr-4r+1)r,\\
P'''(a)=&6(1-t)(1-2r)r^2>0.
\end{align*}
The above last inequality implies that $P''(a)\geq P''(0)=2r[2r^2+(t-4)r+1]:=2r\nu(r)$.
Observe that $\nu(0)=1>0$ and $\nu(\frac{1}{2})=\frac{t-1}{2}<0$. There is a unique number $r_{2,t}=\frac{4-t-\sqrt{t^2-8t+8}}{4}\in (0,\frac{1}{2})$ such that $\nu(r_{2,t})=0$. Hence, $\nu(r)\geq0$ holds for $r\leq r_{2,t}$.
Then we have $P''(a)\geq P''(0)\geq0$ holds for $r\leq\min\{r_{1,t},r_{2,t}\}=r_{1,t}$. The last equality holds because of Lemma \ref{lemma2.6}.
It follows that $P'(a)\geq P'(0)=(1-2r)(tr^2-2r+1)>0$. Thus $P(a)\leq P(1)=0$ holds for $r\leq r_{1,t}$.

{\bf Case 2}  $t\in(t_{*},1)$. The prove is similar to that of  Case 1, we omit it.

\end{proof}

\medskip

{\bf Declarations}\\

{\bf Availability of data}\quad   The data that support the findings of this study are included in this
paper.

{\bf Conflict of interest} \quad The authors have not disclosed any competing interests.

\end{document}